\documentclass{article}%for arXiv

\usepackage{amsmath,amssymb,amsthm}
\usepackage[all]{xy}
\usepackage{fullpage}
\usepackage{amsxtra}

\newtheorem{thm}{Theorem}[section]

\newtheorem{lemma}[thm]{Lemma}
\newtheorem{proposition}[thm]{Proposition}
\newtheorem{corollary}[thm]{Corollary}
\newtheorem{conjecture}[thm]{Conjecture}

\theoremstyle{definition}
\newtheorem{dfn}[thm]{Definition}
\newtheorem{remark}[thm]{Remark}

\newtheorem{example}[thm]{Example}
\newtheorem*{notation}{Notation}

\newcommand{\coker}[1]{\mathrm{coker}\ #1}
\newcommand{\Hom}{\mathrm{Hom}}
\newcommand{\Fil}{\mathrm{Fil}}
\newcommand{\id}{\mathrm{id}}
\newcommand{\colim}{\mathrm{colim}}

\newcommand{\s}{\!}

\begin{document}

\title{New cases of Dwork's conjecture on asymptotic behaviors of solutions of $p$-adic differential equations without solvability}
\author{Shun Ohkubo}
%\author{Shun Ohkubo\thanks{Graduate School of Mathematics, Nagoya University, Furocho, Chikusaku, Nagoya 464-8602, Japan. E-mail address: shun.ohkubo@gmail.com 2020 Mathematics Subject Classification: Primary 12H25; secondary 12H05. Keywords: $p$-adic differential equations, logarithmic growth.}}
\date{\today}
\maketitle

\begin{abstract}
One of the phenomena peculiar in the theory of $p$-adic differential equations is that solutions $f$ of $p$-adic differential equations defined on open discs may satisfy growth conditions at the boundaries. This phenomenon is first studied by Dwork, who proves the fundamental theorem asserting that if a $p$-adic differential equation defined on an open unit disc is solvable, then any solution $f$ has order of logarithmic growth at most $m-1$. In this paper, we study a conjecture proposed by Dwork on a generalization of this theorem to the case without solvability. We prove new cases of Dwork's conjecture by combining descending techniques of differential modules with the author's previous result on Dwork's conjecture in the rank $2$ case.
\end{abstract}

\section{Introduction}
This paper is a continuation of \cite{Ohkc}, where we study Dwork's conjecture concerning asymptotic behaviors of solutions of $p$-adic differential equations defined on the open unit disc.

Throughout this paper, let $K$ denote a complete discrete valuation field of mixed characteristic $(0,p)$ whose valuation $|\ |$ is non-trivial. We define the ring $K\{t\}$ (resp. $K[\s[t]\s]_0$) as the ring of formal power series over $K$ converging (resp. bounded) on the open unit disc $|t|<1$. Let $M$ be a finite free differential module over $K[\s[t]\s]_0$ of rank $m$. Let $H^0(M\otimes K\{t\})$ denote the $K$-vector space of horizontal sections of the differential module $M\otimes K\{t\}$ over $K\{t\}$. Then, we have $\dim_K H^0(M\otimes K\{t\})\le m$, and we say that $M$ is {\it solvable} (in $K\{t\}$) if an equality holds. There are important examples of solvable finite free differential modules over $K[\s[t]\s]_0$; for example, those coming from geometry via the constructions of Picard-Fuchs modules, such as the $p$-adic Gaussian hypergeometric differential equation with parameter $(1/2,1/2,1)$ (\cite[Appendix]{pde}); more generally, those admitting Frobenius structures. In the case where $M$ is solvable, Dwork proves the following fundamental theorem on asymptotic behaviors of elements of $H^0(M\otimes K\{t\})$.

We endow $K\{t\}$ with the increasing filtration $\{K[\s[t]\s]_{\delta}\}_{\delta\in (-\infty,+\infty)}$ of $K[\s[t]\s]_0$-submodules defined by $K[\s[t]\s]_{\delta}=\{\sum_{i\in\mathbb{N}}a_it^i\in K\{t\};\sup_{i\in\mathbb{N}}|a_i|/(i+1)^{\delta}\le +\infty\}$ if $\delta\ge 0$, and $K[\s[t]\s]_{\delta}=0$ if $\delta<0$ (note that $K[\s[t]\s]_0$ as before coincides with $K[\s[t]\s]_0$ here). We endow $H^0(M\otimes K\{t\})$ with the increasing filtration $\{\Fil_{\delta}H^0(M\otimes K\{t\})\}_{\delta\in (-\infty,+\infty)}$ of $K$-subspaces defined by $\Fil_{\delta}H^0(M\otimes K\{t\})=H^0(M\otimes K\{t\})\cap (M\otimes K[\s[t]\s]_{\delta})$ for $\delta\in (-\infty,+\infty)$.

\begin{thm}[{Dwork, see \cite[Theorem 3.2.1]{And}}]\label{Dwork theorem}
Let $M,m$ be as above. Assume that $M$ is solvable. Then, $\Fil_{m-1}H^0(M\otimes K\{t\})=H^0(M\otimes K\{t\})$.
\end{thm}

We mention recent advances in the study of the filtration $\{\Fil_{\delta}H^0(M\otimes K\{t\})\}_{\delta\in (-\infty,+\infty)}$. Chiarellotto and Tsuzuki propose a conjecture on the filtration above when $M$ admits a Frobenius structure. The conjecture asserts that under a technical hypothesis on $M$ called {\it pure of bounded quotient}, the filtration above coincides with the Frobenius slope filtration of $H^0(M\otimes K\{t\})$, which is a local generalization of Dwork's earlier result on the $p$-adic Gaussian hypergeometric differential equation with parameter $(1/2,1/2,1)$.
This conjecture is proved by themselves in the case of $m\le 2$, and by the author in full generality (\cite{CT,Ohkc,Ohkd}). It is worth mentioning that a global analogue of the notion of pure of bounded quotient plays an important role in the proof of the minimal slope conjecture proposed (\cite{Tsu2}).

In this paper, we study Dwork's conjecture, which is regarded as a generalization of Theorem \ref{Dwork theorem} to the case where $M$ may not be solvable. Precisely speaking, the following form of Dwork's conjecture is formulated by the author so that it fits into the setup of \cite{CT} rather than that of \cite{Dw}.

\begin{conjecture}[{Dwork, \cite[Conjecture 2]{Dw}, \cite[Conjecture]{Ohkc}}]\label{Dwork conjecture}
Let $M$ be as above. Let $n$ denote the dimension of $H^0(M\otimes K\{t\})$. Then, $\Fil_{n-1}H^0(M\otimes K\{t\})=H^0(M\otimes K\{t\})$.
\end{conjecture}

We record all the known cases on Conjecture \ref{Dwork conjecture} at this point.

\begin{proposition}\label{known result}
Let $M,m$, and $n$ be as in Conjecture \ref{Dwork conjecture}. Then, Conjecture \ref{Dwork conjecture} is true either when $m\le 2$ or when $n=m,0$.
\end{proposition}
\begin{proof}
In the case of $n=m$, the assertion is nothing but Theorem \ref{Dwork theorem}. In the case of $m=0$, we have nothing to prove. In the case of $m=1$, the assertion follows from the previous cases as $n=0,1$. In the case of $m=2$, the assertion is proved in \cite[Main Theorem]{Ohkc}.
\end{proof}

The goal of this paper is to prove the following new cases of Conjecture \ref{Dwork conjecture}.

\begin{thm}\label{main result}
Let $M,m$, and $n$ be as in Conjecture \ref{Dwork conjecture} such that $M\neq 0$.
\begin{enumerate}
\item Assume that $\lim_{\rho\to 1-0}R_{m-n}(M\otimes K(t)_{\rho})<1$, where $R_{m-n}(M\otimes K(t)_{\rho})$ denotes the $(m-n)$-th subsidiary generic radius of convergence of $M\otimes K(t)_{\rho}$ (\cite[Definition 9.8.1]{pde}). Then, Conjecture \ref{Dwork conjecture} for $M$ is true.
\item Assume that $n=m-1$. Then, Conjecture \ref{Dwork conjecture} for $M$ is true.
\end{enumerate}
\end{thm}

We briefly explain the strategy of the proof of Theorem \ref{main result}. After technical preparations in \S\S \ref{ring}, \ref{differential}, we prove Part (i) of Theorem \ref{main result} in \S \ref{proof} by constructing a certain submodule $L$ of $M$ of free of rank $m-1$ stable under $D$, which is solvable, and then applying Theorem \ref{Dwork theorem} to $L$. To construct such an $L$ under an assumption as in Part (i) of Theorem 1.4, we extend the method of \cite{Ohkc}, which works only in the case of $(m,n)=(2,1)$, to the case of an arbitrary $(m,n)$. Precisely speaking, this is done by using a trick involving exterior powers, which is analogous to the one in the theory of $\varphi$-modules over the Robba ring (\cite{Doc}). Part (ii) of Theorem \ref{main result} is an immediate consequence of Part (i) and a variant of Dwork transfer theorem. 

Finally, we explain our motivation to study Conjecture \ref{Dwork conjecture}. In the theory of $G$-functions, we study arithmetic aspects of ordinary linear differential equations of the form $d^mf/dt^m+p_{m-1}d^{m-1}f/dt^{m-1}+\dots+p_0f=0$ for $p_{m-1},\dots,p_0\in\mathbb{Q}(t)$. To such differential equations, we can associate differential modules $N$ over $\mathbb{Q}_p[\s[t]\s]_0$ via appropriate base changes. As the conjecture of Chiarellotto and Tsuzuki suggests, it is natural to ask whether the filtration $\{\Fil_{\delta}H^0(N\otimes \mathbb{Q}_p\{t\})\}_{\delta\in (-\infty,+\infty)}$ has arithmetic information about the given differential equation. To study this filtration, Theorems \ref{Dwork theorem} may not be applicable since $N$ may not be solvable by Dwork-Frobenius lemma (\cite[Lemma IV 2.1]{G}). Hence, the author thinks that we should prove general properties of the filtration $\{\Fil_{\delta}H^0(M\otimes K\{t\})\}_{\delta\in (-\infty,+\infty)}$ valid even when $M$ is not solvable, such as Conjecture \ref{Dwork conjecture}.

\section*{Acknowledgement}
This work is supported by JSPS KAKENHI Grant-in-Aid for Scientific Research (C) 22K03227 and Grant-in-Aid for Scientific Research (S) 24H00015.

\section{Rings of formal Laurent series}\label{ring}
In this section, we recall the definitions and properties of rings used in the theory of $p$-adic differential equations, mainly following \cite[\S 8]{pde}.

Throughout this paper, a {\it ring} means a commutative associative ring. Furthermore, for a ring $R$, an {\it $R$-module} means a left $R$-module, and a {\it derivation} $d$ {\it on} $R$ means a derivation $d:R\to R$ in the usual sense. A {\it formal Laurent series} over $K$ means a formal sum $\sum_{i\in\mathbb{Z}}a_it^i$ with $a_i\in K$ for $i\in\mathbb{Z}$.

Let $\alpha,\beta$ be elements of $[0,+\infty)$ such that $\alpha\le\beta$ and $(\alpha,\beta)\neq (0,0)$. We define the ring $K\langle\alpha/t,t/\beta\rangle$ for any $\alpha,\beta$, and the rings $K\langle\alpha/t,t/\beta]\s]_0,K\langle\alpha/t,t/\beta\}$ when $\alpha\neq\beta$. When $\alpha\neq 0$, we define $K\langle\alpha/t,t/\beta\rangle$ as the set of formal Laurent series $\sum_{i\in\mathbb{Z}}a_it^i$ satisfying $|a_i|\alpha^i\to 0$ as $i\to-\infty$ and $|a_i|\beta^i\to 0$ as $i\to+\infty$. Furthermore, when $\alpha\neq 0$, we define $K\langle\alpha/t,t/\beta]\s]_0$ (resp. $K\langle\alpha/t,t/\beta\}$) as the set of formal Laurent series $\sum_{i\in\mathbb{Z}}a_it^i$ satisfying $|a_i|\alpha^i \to 0$ as $i\to-\infty$ and $\sup_{i\ge 0}|a_i|\beta^i<+\infty$ (resp. $|a_i|\alpha^i \to 0$ as $i\to-\infty$ and $|a_i|\gamma^i\to 0$ as $i\to+\infty$ for $\gamma\in (0,\beta))$. When $\alpha=0$, we define $K\langle\alpha/t,t/\beta\rangle$ (resp. $K\langle\alpha/t,t/\beta]\s]_0,K\langle\alpha/t,t/\beta\}$) as the set of formal power series $\sum_ia_it^i$ satisfying $|a_i|\beta^i\to 0$ as $i\to+\infty$ (resp. $\sup_{i\ge 0}|a_i|\beta^i<+\infty$, $|a_i|\gamma^i\to 0$ as $i\to+\infty$ for $\gamma\in (0,\beta))$. In all cases, the multiplications are given by convolution. For simplicity, we denote $K\langle 0/t,t/\beta\rangle,K\langle 0/t,t/\beta]\s]_0$, and $K\langle 0/t,t/\beta\}$ by $K\langle t/\beta\rangle,K[\s[t/\beta]\s]_0$, and $K\{t/\beta\}$ respectively, which are regarded as subrings of the ring $K[\s[t]\s]$ of formal power series over $K$. Moreover, we denote $K\langle t/1\rangle,K[\s[t/1]\s]_0$, and $K\{t/1\}$ by $K\langle t\rangle,K[\s[t]\s]_0$, and $K\{t\}$ respectively. 

Let $\rho\in (0,+\infty)$. We define the ring $K(t)_{\rho}$ as the completion of $K(t)$ with respect to $\rho$-Gauss valuation. Also, define the {\it Amice ring} $\mathcal{E}$ over $K$ as the set of formal Laurent series $\sum_{i\in\mathbb{Z}}a_it^i$ such that $\sup_{i\ge 0}|a_i|<+\infty$, and $|a_i|\to 0$ as $i\to-\infty$ with multiplication given by convolution. Both $K(t)_{\rho}$ and $\mathcal{E}$ are complete non-archimedean valuation fields of mixed characteristic $(0,p)$, where the valuation on $\mathcal{E}$ is given by $1$-Gauss valuation.

We have the chains of rings $K\langle\alpha/t,t/\beta\rangle\subset K\langle\alpha/t,t/\beta]\s]_0\subset K\langle\alpha/t,t/\beta\}$, and $K\langle\alpha'/t,t/\beta'\rangle\subset K\langle\alpha/t,t/\beta\rangle$, $K\langle\alpha/t,t/\beta]\s]_0\subset K\langle\alpha'/t,t/\beta']\s]_0$, and $K\langle\alpha/t,t/\beta\}\subset K\langle\alpha'/t,t/\beta'\}$ for $\alpha',\beta'\in [0,+\infty)$ such that $\alpha\le\alpha'\le\beta'\le \beta$. We have a natural injective ring homomorphism $K\langle\alpha/t,t/\beta\rangle\to K(t)_{\rho}$ for $\rho\in [\alpha,\beta)$. We also have natural injective ring homomorphisms $K\langle\alpha/t,t/\beta]\s]_0\to K(t)_{\rho},K\langle\alpha/t,t/\beta]\s]_0\to K(t)_{\rho}$ for $\rho\in (\alpha,\beta)$. Furthermore, when $\alpha\in [0,1)$, we have a natural injective ring homomorphism $K\langle\alpha/t,t]\s]_0\to\mathcal{E}$.

For a formal power (resp. Laurent) series $x=\sum_{i\in\mathbb{N}}a_it^i$ (resp. $\sum_{i\in\mathbb{Z}}a_it^i$), we denote the formal power (resp. Laurent) series $\sum_{i\in\mathbb{N}}(i+1)a_{i+1}t^i$ (resp. $\sum_{i\in\mathbb{Z}}(i+1)a_{i+1}t^i$) by $d(x)$. Then, we define the derivations $d$ on the rings $K[\s[t]\s],K\langle\alpha/t,t/\beta\rangle$, $K\langle\alpha/t,t/\beta]\s]_0$, and $K\langle\alpha/t,t/\beta\}$ by $x\mapsto d(x)$. We similarly define derivations $d$ on $K(t)_{\rho},\mathcal{E}$.

\begin{lemma}\label{unit lemma}
For $\alpha,\beta\in [0,+\infty)$ such that $\alpha<\beta$, we have $K\langle\alpha/t,t/\beta\}=\cap_{\beta'\in (\alpha,\beta)}K\langle\alpha/t,t/\beta'\rangle$, and $K\langle\alpha/t,t/\beta\}^{\times}=K\langle\alpha/t,t/\beta]\s]_0^{\times}$.
\end{lemma}
\begin{proof}
The first assertion follows by definition. To prove the second assertion, it suffices to prove that $K\langle\alpha/t,t/\beta\}^{\times}\subset K\langle\alpha/t,t/\beta]\s]_0$. We first consider the case of $\alpha\neq 0$. Recall that the function $|\ |_{\gamma}:K\langle\alpha/t,t/\beta\}\to \mathbb{R}_{\ge 0};\sum_{i\in\mathbb{Z}}a_it^i\mapsto \sup_i|a_i|\gamma^i$ for $\gamma\in [\alpha,\beta)$ is a valuation (\cite[Definition 8.2.1]{pde}). We suppose, by way of contradiction, that there exists $x=\sum_{i\in\mathbb{Z}}a_it^i\in K\langle\alpha/t,t/\beta\}^{\times}$ such that $x\notin K\langle\alpha/t,t/\beta]\s]_0$. Then, we have $\sup_{i<0}|a_i|\beta^i\le \sup_{i<0}|a_i|\alpha^i\le |x|_{\alpha}<+\infty$. Hence, $\sup_{i\ge 0}|a_i|\beta^i=+\infty$, which implies that $|x|_{\gamma}\to+\infty$ as $\gamma\to \beta-0$. Hence, we have $|y|_{\gamma}=1/|x|_{\gamma}\to 0$ as $\gamma\to \beta-0$. Since $|a_i|\gamma^i\le |y|_{\gamma}$ for $i\in \mathbb{Z}$ and $\gamma\in [\alpha,\beta)$, we have $|a_i|\beta^i=0$ for $i\in \mathbb{Z}$. Hence, $a_i=0$ for $i\in \mathbb{Z}$, i.e., $x=0$, which is a contradiction.

In the case of $\alpha=0$, we have $K\{t/\beta\}^{\times}\subset K\langle (\beta/2)/t,t/\beta\}^{\times}\cap K\{t/\beta\}\subset K\langle (\beta/2)/t,t/\beta]\s]_0\cap K\{t/\beta\}=K[\s[t/\beta]\s]_0$ by the previous case.
\end{proof}

We recall some definitions in the ring theory. Let $R$ be a ring. We say that $R$ is an {\it elementary divisor ring} if, for any matrix $A\in M_{mn}(R)$, there exist $U\in GL_m(R),V\in GL_n(R)$ and elements $d_1,\dots,d_{\min\{m,n\}}\in R$ satisfying $d_1|\dots|d_{\min\{m,n\}}$, such that $UAV$ is equal to the diagonal matrix in $M_{mn}(R)$ whose diagonal entries are given by $d_1,\dots,d_{\min\{m,n\}}$; the elements $d_1,\dots,d_{\min\{m,n\}}$ are called the {\it elementary divisors} of $A$. We say that $R$ is {\it B\'ezout} (resp. {\it B\'ezout domain}) if any finitely generated ideal of $R$ is principal (resp. and  $R$ is an integral domain). We say that $R$ is an {\it elementary divisor domain} if $R$ is both an elementary divisor ring and an integral domain. Finally, we say that $R$ is {\it adequate} if it is a B\'ezout integral domain, and, for any $a,b\in R$, there exist $a_1,a_2\in R$ such that $a=a_1a_2$, $(a_1,b)=R$, and, for any $a_3|a_2$ with $a_3\notin R^{\times}$, we have $(a_3,b)\neq R$. Recall that if $R$ is adequate, then $R$ is an elementary divisor domain (\cite[Theorem 3]{Hel}). Also, recall that if $R$ is B\'ezout, then an $R$-module $M$ is flat if and only if $M$ is torsion free. Hence, if $R,R'$ are a B\'ezout domain and an integral domain respectively, then any injective ring homomorphism $f:R\to R'$ is flat.

\begin{lemma}\label{adequate}
The rings $K\langle\alpha/t,t/\beta\rangle,K\langle\alpha/t,t/\beta]\s]_0$, and $K\langle\alpha/t,t/\beta\}$ for $\alpha,\beta$ as above are elementary divisor domains.
\end{lemma}
\begin{proof}
The rings $K\langle\alpha/t,t/\beta\rangle,K\langle\alpha/t,t/\beta]\s]_0$ are principal ideal domains (\cite[Lemma 9.1.1]{pde}; recall that we assume that $K$ is discretely valued). The ring $K\langle\alpha/t,t/\beta\}$ is adequate, which follows from Lazard's results (\cite{Laz}, see also \cite[Corollaire 4.30]{Chr2}) as in \cite[Proposition 4.12]{Ber}.
\end{proof}

\section{Differential rings and differential modules}\label{differential}
In this section, we recall basic terminology of differential rings and differential modules following \cite[\S 5]{pde}. We also prove that the category of finite free differential modules over a differential ring $R$ is abelian if $R$ has a nice property, which is called {\it Property} $(P)$ in this paper.

A {\it differential ring} is a ring $R$ equipped with a derivation $d$ on $R$, which is denoted by $R$ for simplicity. In the rest of this section, let $R$ be a differential ring unless otherwise is mentioned. A {\it differential module} over $R$ is an $R$-module $M$ equipped with an endomorphism $D$ on $M$ as an abelian group, called a {\it differential operator}, satisfying $D(rx)=d(r)x+rD(x)$ for $r\in R,x\in M$, which is denoted by $M$ for simplicity. 
We also denote by $R$ the differential module given by $R$ equipped with the differential operator given by $d$. A differential module $M$ over $R$ is called {\it finite free of rank $m$} (resp. {\it flat}) if so is the underlying $R$-module of $M$. We define the category of differential modules over $R$ by defining a morphism of differential modules $M\to N$ as a morphism of $R$-modules $M\to N$ commuting with differential operators. This category is equipped with the forgetful functor to the category of $R$-modules, which is exact, faithful, and preserves kernels and cokernels. Moreover, some bifunctors and endofunctors of the category of $R$-modules naturally extends to the category of differential modules over $R$; for example, the tensor product ${-}\otimes{-}$, the dual $(-)\spcheck:=\Hom(-,R)$, the $n$-fold tensor product ${-}^{\otimes n}$ for $n\in\mathbb{N}$, and the $n$-th exterior power $\wedge^n -$ for $n\in\mathbb{N}$.

We denote by $R'$ the subring of $R$ given by the kernel of $d$. For a differential module $M$ over $R$, we denote by $H^0(M)$ the $R'$-module given by the kernel of $D:M\to M$.

We define the category of differential rings by defining a morphism $R\to S$ of differential rings as a ring homomorphism $f:R\to S$ commutes with derivations. We can define the pullback functor ${-}\otimes S$ from the category of differential modules over $R$ to that over $S$ so that it is compatible with the usual pullback functor from the category of $R$-modules to that of $S$-modules.

\begin{dfn}
We define {\it Conditions $(S)$} for a differential ring $R$ by
\begin{center}
$(S)$ if $I$ is a principal ideal of $R$ such that $d(I)\subset I$, then $I=(0),(1)$.
\end{center}
We say that a differential ring $R$ has {\it Property $(P)$} if $R$ is an elementary divisor domain satisfying Condition $(S)$. Note that one can find a condition similar to Condition $(S)$ in \cite[Corollaire 6.21]{Chr2}.
\end{dfn}

\begin{remark}\label{remark on rings}
\begin{enumerate}
\item Condition $(S)$ is equivalent to saying that if $r\in R$ satisfies $d(r)\in Rr$, then we have either $r=0$ or $r\in R^{\times}$. Consequently, under Condition $(S)$, $R'$ is a subfield of $R$.
\item Let $f:R\to S$ be a morphism of differential rings. Assume that $f$ is an injection, $R^{\times}\to S^{\times};x\mapsto f(x)$ is a bijection, and $S$ satisfies Condition $(S)$. Then, by Part (i), $R$ satisfies Condition $(S)$.
\item Let $C$ denote the category of differential rings. We can easily see that $C$ is complete and cocomplete. Let $\Lambda$ be a small category, and $F:\Lambda\to C$ a functor. If $F(\lambda)$ for $\lambda\in\Lambda$ satisfies Condition $(S)$, then so do $\lim{F}$ and $\colim{F}$, which follows from Part (i).
\end{enumerate}
\end{remark}

\begin{lemma}
\begin{enumerate}
\item The differential ring $K[\s[t]\s]$ has Property $(P)$.
\item The differential rings $K\langle\alpha/t,t/\beta\rangle,K\langle\alpha/t,t/\beta]\s]_0$, and $K\langle\alpha/t,t/\beta\}$ for $\alpha,\beta$ as in \S \ref{ring} have Property $(P)$.
\end{enumerate}
\end{lemma}
\begin{proof}
Part (i) is obvious. To prove Part (ii), by Lemma \ref{adequate}, we have only to verify Condition $(S)$. In the case of $K\langle\alpha/t,t/\beta\rangle$, let $I$ be a principal ideal of $K\langle\alpha/t,t/\beta\rangle$ such that $d(I)\subset I$. By a consequence of Weierstrass preparation (\cite[Proposition 8.3.2]{pde}), we have either $I=(0)$ or $I=(P)$ for $P\in K[t]$ whose zeroes in an algebraic closure $K^{\mathrm{alg}}$ of $K$ have valuations in $(\alpha,\beta)$. In the latter case, we have $d(P)=QP$ for some $Q\in K\langle\alpha/t,t/\beta\rangle$. By calculating orders of zeroes, we can see that $P$ does not have zeroes in $K^{\mathrm{alg}}$, which implies that $I=(1)$. In the case of $K\langle\alpha/t,t/\beta\}$, the assertion follows from the previous case, Part (iii) of Remark \ref{remark on rings}, and Lemma \ref{unit lemma}. In the case of $K\langle\alpha/t,t/\beta]\s]_0$, the assertion follows from the last case, Part (ii) of Remark \ref{remark on rings}, and Lemma \ref{unit lemma}.
\end{proof}

\begin{lemma}\label{abelian category}
Let $R$ be a differential ring having Property $(P)$. 
\begin{enumerate}
\item For any morphism $\varphi:M\to N$ of finite free differential modules over $R$, the kernel and cokernel of $\varphi$ are again finite free. Consequently, the category of finite free differential modules over $R$ is an abelian subcategory of the category of differential modules over $R$.
\item Let $\varphi$ be as in Part (i). If $\varphi$ is an injection, and the ranks of $M,N$ coincide, then $\varphi$ is an isomorphism.
\item Let $\varphi$ be as in Part (i). If $\varphi$ is an injection, then so is $\wedge^n\varphi:\wedge^n M\to\wedge^nN$ for $n\in\mathbb{N}$.
\end{enumerate}
\end{lemma}
\begin{proof}
We prove Part (i). We may assume that $\varphi\neq 0$. Let $m,n$ denote the ranks of $M,N$ respectively, $A\in M_{mn}(R)$ the matrix of $\varphi$ with respect to any chosen basis of $M,N$. Let $d_1,\dots,d_{\min\{m,n\}}$ denote the elementary divisors of $A$. Then, we have isomorphisms $\ker{\varphi}\cong R^{m-k},\coker{\varphi}\cong \oplus_{i=1}^kR/Rd_i\oplus R^{n-k}$ of $R$-modules, where $k\in\{1,\dots,\min\{m,n\}\}$ denotes the unique element satisfying $d_k\neq 0,d_{k+1}=\dots=d_{\min\{m,n\}}=0$. Hence, it suffices to prove that $d_k\in R^{\times}$. Let $I$ denote the ideal of $R$ given by $\{r\in R;\forall x\in \coker{\varphi},rx=0\}$. Then, we have $d(I)\subset I$ by definition, and $I=Rd_k$. Hence, by Property $(P)$, $I=(1)$, which implies the assertion.

We prove Part (ii). Let $F$ denote the fraction field of $R$. By Part (i), it suffices to prove that $\coker{\varphi}\otimes F=0$. The morphism $\varphi\otimes \id_F:M\otimes F\to N\otimes F$ is an isomorphism by assumption. Hence, $\coker{(\varphi\otimes \id_F)}\cong\coker{\varphi}\otimes F=0$.

We prove Part (iii). In the category of $R$-modules, since $\varphi$ is a coretraction, that is, there exists $\psi:N\to M$ such that $\psi\circ\varphi=\id_M$, by Part (i), so is $\wedge^n\varphi$, which implies the assertion.
\end{proof}

Finally, we recall the definitions of the multiset of subsidiary generic radii of convergence and its variation (\cite[Definitions 9.4.4, 9.8.1, Notation 11.3.1, and Remarks 11.3.4, 11.6.5]{pde}). Let $M$ be a non-zero finite free differential module over a differential field $F$ of rank $m$. Then, $M$ is a successive extension of irreducible finite free differential modules $M_1,\dots,M_n$ over $F$. We define the multiset of {\it subsidiary generic radii of convergence} of $M$ by the multiset consisting of $|p|^{1/(p-1)}/|D|_{\mathrm{sp},M_i}$ with multiplicity $\dim_F{M_i}$ for $i=1,\dots,m$, where $|D|_{\mathrm{sp},M_i}$ denotes the spectral radius of $D:M_i\to M_i$ with respect the supremum norm associated to any basis of $M_i$. 

Let $R$ denote any one of the differential rings $K\langle\alpha/t,t\rangle,K\langle\alpha/t,t]\s]_0$, and $K\langle\alpha/t,t\}$ for $\alpha\in [0,1)$, and $I$ denote $[\alpha,1],[\alpha,1]$, and $[\alpha,1)$ respectively. Let $M$ be a non-zero finite free differential module over $R$ of rank $m$. We define the multiset $R_1(M,\rho)\le \dots\le R_m(M,\rho)$ for $\rho\in I$ as follows.  For $\rho\in [\alpha,1)$, or $\rho=1$ when $R=K\langle\alpha/t,t\rangle$, we define the multiset $R_1(M,\rho)\le \dots\le R_m(M,\rho)$ by $R_1(M\otimes K(t)_{\rho})\le\dots\le R_m(M\otimes K(t)_{\rho})$. For $\rho=1$ when $R=K\langle\alpha/t,t]\s]_0$, we define $R_1(M,1)\le \dots\le R_m(M,1)$ by $R_1(M\otimes \mathcal{E})\le\dots\le R_m(M\otimes \mathcal{E})$. Finally, we set $F_i(M,r)=-\log R_1(M,e^{-r})-\dots--\log R_1(M,e^{-r})$ for $i=1,\dots,m$ and $r\in [0,+\infty)$ such that $e^{-r}\in I$.

\section{Proof of Theorem \ref{main result} }\label{proof}
\begin{notation}
Let $M$ be a finite free differential module over $K[\s[t]\s]_0$ of rank $m$. In the rest of this paper, as in Theorem \ref{main result}, we assume that $M\neq 0$, i.e., $m>0$ . We identify $K\{t\}'$ as $K$ via the obvious isomorphism $K\to K\{t\}'$. Let $n$ denote the dimension of the $K$-vector space $H^0(M\otimes K\{t\})$. Let $P$ denote $H^0(M\otimes K\{t\})\otimes K\{t\}$, which is a finite free differential module over $K\{t\}$ of rank $n$ with respect to the differential operator given by $\id_{H^0(M\otimes K\{t\})}\otimes d$. Let $\Phi:P\to M\otimes K\{t\}$ denote the obvious morphism.
\end{notation}

\begin{dfn}\label{D}
We say that $M$ satisfies {\it Condition} $(D)$ if there exists a $3$-tuple $(L,\varphi,\theta)$, where $L$ is a finite free differential module over $K[\s[t]\s]_0$ of rank $n$, $\varphi$ is an injective morphism $L\to M$, and $\theta$ is an isomorphism $L\otimes K\{t\}\to P$, such that the diagram
$$ \xymatrix{L\otimes K\{t\}\ar[r]^{\varphi\otimes \id_{K\{t\}}}\ar[d]^{\theta}&M\otimes K\{t\} \ar[d]^{\id_{M\otimes K\{t\}}} \\ P\ar[r]^(.4){\Phi} & M\otimes K\{t\}} $$
is commutative.
\end{dfn}

\begin{remark}\label{remark on D}
Note that when $n=0,m$, Condition $(D)$ for $M$ holds as the $3$-tuples $(0,0,0),(M,\id_M,\Phi^{-1})$ respectively satisfy the requirements of Condition $(D)$. In particular, Condition $(D)$ for $M$ holds when $m=1$. Moreover, as we will prove in Corollary \ref{two}, Condition $(D)$ for $M$ holds even when $m=2$.
\end{remark}

\begin{example}[{\cite[Appendix 2]{Ohkc}}]\label{example}
Assume that $p\neq 2$. Let $M$ denote the finite free differential module over $K[\s[t]\s]_0$ of rank $2$ with a basis $\{e_1,e_2\}$ such that $D(e_1)=e_2,D(e_2)=-e_1-te_2$. Then, $H^0(M\otimes K\{t\})=K(te_1+e_2)$, in particular, $n=1$. Moreover, $M$ satisfies Condition $(D)$ as the $3$-tuple $(L,\varphi,\theta)$ defined by $L=K[\s[t]\s]_0$, $\varphi:L\to M;1\mapsto te_1+e_2$, and $\theta:L\otimes K\{t\}\to H^0(M\otimes K\{t\})\otimes K\{t\};1\otimes 1\mapsto (te_1+e_2)\otimes 1$ satisfies the requirements of Condition $(D)$.

We prove that $\Phi$ is not a coretraction. We suppose, by way of contradiction, that $\Phi$ is a coretraction. Then, $M\otimes K\{t\}$ has a quotient isomorphic to $K\{t\}$. Hence, $(M\otimes K\{t\})\spcheck$ has a subobject isomorphic to $K\{t\}\spcheck\cong K\{t\}$, which implies that $H^0(M\spcheck\otimes K\{t\})\cong H^0((M\otimes K\{t\})\spcheck)\neq 0$. This contradicts to the fact that $H^0(M\spcheck\otimes K\{t\})=0$ (\cite[p. 1553]{Ohk}).
\end{example}

\begin{proposition}\label{key observation}
Assume that Condition $(D)$ for $M$ holds. Then, Conjecture \ref{Dwork conjecture} for $M$ is true.
\end{proposition}
\begin{proof}
Let notation be as in Definition \ref{D}. Then, we have the commutative diagram
$$ \xymatrix{ \Fil_{n-1}H^0(L\otimes K\{t\})\ar[r]\ar[d]&\Fil_{n-1}H^0(M\otimes K\{t\})\ar[d] \\  H^0(L\otimes K\{t\})\ar[r]&H^0(M\otimes K\{t\}),} $$
where the vertical (resp. horizontal) arrows are the inclusions (resp. induced by $\varphi$). The bottom horizontal arrow is an isomorphism since the $K$-vector space $H^0(L\otimes K\{t\})$ has dimension $n$ by $H^0(L\otimes K\{t\})\cong H^0(P)\cong H^0(M\otimes K\{t\})$. By applying Theorem \ref{Dwork theorem} to $L$, the left vertical arrow is an isomorphism. Hence, the right vertical arrow is an isomorphism.
\end{proof}

We briefly explain the idea of the proof of Part (i) of Theorem \ref{main result}. By Proposition \ref{key observation}, it suffices to verify Condition $(D)$ for $M$ as in Part (i) of Theorem \ref{main result}. To achieve this goal, we need to construct a subobject $L$ of $M$. Unfortunately, such an $L$, if exists, is not necessary indecomposable as we see in Example \ref{example}, which implies that the idea to exploit decomposition results over $K\{t\}$, such as \cite[Theorem 12.5.1]{pde}, may not be effective. Alternatively, we will exploit a decomposition of $M\otimes K\langle\alpha/t,t]\s]_0$ for some $\alpha\in (0,1)$ to achieve our goal.

\begin{thm}[{Theorem \ref{main result} (i)}]\label{key lemma}
Assume that $R_{m-n}(M,1)<1$. Then, Condition $(D)$ for $M$ holds.
\end{thm}
\begin{proof}
For $\rho\in (0,1)$, we have $R_1(P,\rho)=\dots=R_n(P,\rho)=\rho$ by $P\otimes K(t)_{\rho}\cong H^0(M\otimes K\{t\})\otimes K\{t\}\otimes K(t)_{\rho}\cong K(t)_{\rho}^n$. Hence, we have $R_{m-n+1}(M,\rho)=\dots=R_m(M,\rho)=\rho$ for $\rho\in (0,1)$. By \cite[Theorem 11.3.2, Remark 11.6.5]{pde}, $R_i(M,1)=\lim_{\rho\to 1-0}R_i(M,\rho)=1$ for $i=m-n+1,\dots,m$, i.e., $R_{m-n+1}(M,1)=\dots=R_m(M,1)=1$. By \cite[Theorem 2.3.9, Remark 2.3.11]{KX} (see \cite[Theorem 3.5]{Ohkc} for a detailed proof), there exists a $4$-tuple $(\alpha,Q,Q',\chi)$, where $\alpha\in (0,1)$, $Q,Q'$ are finite free differential modules over $K\langle\alpha/t,t]\s]_0$ of ranks $n,m-n$ respectively such that $R_i(Q',\rho)=R_i(M,\rho)$ for $\rho\in [\alpha,1]$ and $i=1,\dots,m-n$, and $\chi$ is an isomorphism $Q\oplus Q'\to M\otimes K\langle\alpha/t,t]\s]_0$.

We prove the claim that $\Hom(P\otimes K\langle\alpha/t,t\},Q'\otimes K\langle\alpha/t,t\})=0$. Fix $\rho'\in (\max\{R_{m-n}(M,1),\alpha\},1)$. Since $\lim_{\rho\to 1-0}R_{m-n}(M,\rho)=R_{m-n}(M,1)<\rho'$ by \cite[Theorem 11.3.2]{pde}, there exists $\alpha'\in [\rho',1)$ such that $R_{m-n}(M,\rho)<\rho'$ for $\rho\in [\alpha',1)$. Fix $\rho\in (\alpha',1)$. Then, the subsidiary generic radii of convergence of $P\otimes K(t)_{\rho},Q\otimes K(t)_{\rho}$ are disjoint by $R_1(Q',\rho)\le\dots\le R_{m-n}(Q',\rho)=R_{m-n}(M,\rho)<\rho'<\rho=R_1(P,\rho)$. Hence, we have $\Hom(P\otimes K\langle\alpha/t,t\}\otimes K(t)_{\rho},Q'\otimes K\langle\alpha/t,t\}\otimes K(t)_{\rho})\cong \Hom(P\otimes K(t)_{\rho},Q'\otimes K(t)_{\rho})=0$. Since $P,Q'$ are finite free, the obvious morphism $\Hom(P\otimes K\langle\alpha/t,t\},Q'\otimes K\langle\alpha/t,t\})\to \Hom(P\otimes K\langle\alpha/t,t\}\otimes K(t)_{\rho},Q'\otimes K\langle\alpha/t,t\}\otimes K(t)_{\rho})$ is an injection, which implies the claim.

Let $\Psi$ denote the morphism $Q\to M\otimes K\langle\alpha/t,t]\s]_0$ obtained by the composition of the obvious morphism $Q\to Q\oplus Q'$ followed by $\chi$. Let $\Phi'$ (resp. $\Psi'$) denote the injective morphism $P\otimes K\langle\alpha/t,t\}\to M\otimes K\langle\alpha/t,t\}$ (resp. $Q\otimes K\langle\alpha/t,t\}\to M\otimes K\langle\alpha/t,t\}$) induced by $\Phi$ (resp. $\Psi$). By the claim, there exists a unique morphism $\Gamma:P\otimes K\langle\alpha/t,t\}\to Q\otimes K\langle\alpha/t,t\}$ making the square in a diagram
$$ \xymatrix{&P\otimes K\langle\alpha/t,t\}\ar[r]^{\Phi'}\ar[d]^{\Gamma} &M\otimes K\langle\alpha/t,t\}\ar[d]^{\id_{M\otimes K\langle\alpha/t,t\}}}&& \\ 0\ar[r]&Q\otimes K\langle\alpha/t,t\}\ar[r]^{\Psi'} & M\otimes K\langle\alpha/t,t\}\ar[r]&Q'\otimes K\langle\alpha/t,t\}\ar[r] &0}  $$
commutative, where the bottom row is exact. By Lemma \ref{abelian category}, $\Gamma$ is an isomorphism. Applying the functor $\wedge^n{-}$ to the square in the diagram above, we obtain the commutative diagram
$$ \xymatrix{\wedge^n (P\otimes K\langle\alpha/t,t\})\ar[r]^{\wedge^n\Phi'}\ar[d]_{\cong}^{\wedge^n\Gamma} &\wedge^n (M\otimes K\langle\alpha/t,t\})\ar[d]^{\id_{\wedge^n (M\otimes K\langle\alpha/t,t\})}} \\ \wedge^n (Q\otimes K\langle\alpha/t,t\})\ar[r]^{\wedge^n\Psi'} & \wedge^n (M\otimes K\langle\alpha/t,t\}),}  $$
where the horizontal arrows are injections by Lemma \ref{abelian category}. Let $\{f_1,\dots,f_n\}$ (resp. $\{g_1,\dots,g_n\}$) be a basis of $P$ (resp. $Q$) such that $D(f_1)=\dots=D(f_n)=0$. Let $f',g'$ denote $\wedge^n\Phi'(f_1\otimes 1\wedge\dots\wedge f_n\otimes 1),\wedge^n\Psi'(g_1\otimes 1\wedge\dots\wedge g_n\otimes 1)\in \wedge^n (M\otimes K\langle\alpha/t,t\})$ respectively. By the last diagram, the image of $\wedge^n\Phi'$ coincides with that of $\wedge^n\Psi'$, which is a finite free differential module over $K\langle\alpha/t,t\}$ of rank $1$ generated by $f'$ as well as by $g'$. Hence, there exists $u\in K\langle\alpha/t,t\}^{\times}=K\langle\alpha/t,t]\s]_0^{\times}$ such that $f'=ug'$. After replacing $g_1$ by $ug_1$, we may assume that $u=1$.

We have an exact sequence
$$\xymatrix{0\ar[r]&K[\s[t]\s]_0\ar[r]&K\{t\}\times K\langle\alpha/t,t]\s]_0\ar[r]&K\langle\alpha/t,t\}\ar[r] &0, }  $$
where the second (resp. third) arrow is induced by (resp. the difference of) the inclusions $K[\s[t]\s]_0\to K\{t\},K[\s[t]\s]_0\to K\langle\alpha/t,t]\s]_0$ (resp. $K\{t\}\to K\langle\alpha/t,t\},K\langle\alpha/t,t]\s]_0\to K\langle\alpha/t,t\}$). By applying the functor $(\wedge^n M)\otimes -$ to the exact sequence, we obtain an exact sequence
$$\xymatrix{0\ar[r]&\wedge^n M\ar[r]^(.2){\tau\times \tau'}&\wedge^n (M\otimes K\{t\})\times \wedge^n (M\otimes K\langle\alpha/t,t]\s]_0)\ar[r]^(.63){\sigma}&\wedge^n(M\otimes K\langle\alpha/t,t\})\ar[r] &0, } $$
where $\tau:\wedge^n M\to \wedge^n (M\otimes K\{t\}),\tau':\wedge^n M\to \wedge^n (M\otimes K\langle\alpha/t,t]\s]_0)$ denote the obvious morphisms. Then, $\sigma(\Phi(f_1)\wedge\dots\wedge \Phi(f_n),\Psi(g_1)\wedge\dots\wedge \Psi(g_n))=f'-g'=0$. Hence, there exists a unique element $e\in \wedge^n M$ such that $\tau\times\tau'(e)=(\Phi(f_1)\wedge\dots\wedge \Phi(f_n),\Psi(g_1)\wedge\dots\wedge \Psi(g_n))$. Then, $D(e)=0$ since $\tau$ is an injection, and $\tau(D(e))=D(\tau(e))=D(\Phi(f_1)\wedge\dots\wedge \Phi(f_n))=\Phi(D(f_1))\wedge\dots\wedge \Phi(f_n)+\dots+\Phi(f_1)\wedge\dots\wedge \Phi(D(f_n))=0$. Hence, the map $M\to \wedge^{n+1}M;x\mapsto x\wedge e$, which is denoted by $\delta$, is a morphism of differential modules over $K[\s[t]\s]_0$. Let $L$ denote the kernel of $\delta$, and $\varphi$ the canonical injection $L\to M$ so that we have the exact sequence
$$ \xymatrix{0\ar[r]&L\ar[r]^{\varphi}&M\ar[r]^(.4){\delta}&\wedge^{n+1}M }$$
in the category of finite free differential modules over $K[\s[t]\s]_0$.

Let $\Delta$ denote the map $M\otimes K\{t\}\to \wedge^{n+1}(M\otimes K\{t\});x\mapsto x\wedge \Phi(f_1)\wedge\dots\wedge \Phi(f_n)$. Then, similarly as before, $\Delta$ is a morphism of differential modules over $K\{t\}$. Since there exists a basis of $M\otimes K\{t\}$ of the form $\{\Phi(f_1),\dots,\Phi(f_n),\dots\}$ by Part (i) of Lemma \ref{abelian category}, the kernel of $\Delta$ coincides with $K\{t\}\Phi(f_1)+\dots+K\{t\}\Phi(f_n)=\Phi(P)$. Hence, we have the commutative diagram with exact rows
$$ \xymatrix{ 0\ar[r]&L\otimes K\{t\}\ar[r]^{\varphi\otimes \id_{K\{t\}}}&M\otimes K\{t\}\ar[r]^(.4){\delta\otimes \id_{K\{t\}}}\ar[d]^{\id_{M\otimes K\{t\}}}&(\wedge^{n+1}M)\otimes K\{t\}\ar[d]_{\cong}\\ 0\ar[r]&P\ar[r]^(.4){\Phi}&M\otimes K\{t\}\ar[r]^(.4){\Delta}& \wedge^{n+1}(M\otimes K\{t\}), }  $$
where the right vertical arrows is an obvious isomorphism. Therefore, there exists a unique morphism $\theta:L\otimes K\{t\}\to P$  such that $\Phi\circ\theta=\varphi\otimes\id_{K\{t\}}$. By construction, the $3$-tuple $(L,\varphi,\theta)$ satisfies the requirements for Condition $(D)$.
\end{proof}

\begin{remark}
The idea of the construction of $e$ (resp, $L,\theta$) is similar to that in the proof of \cite[Lemma 12.2.4]{pde} (resp. \cite[Lemma 3.6.2]{Doc}), which is a result on descents of modules (resp. $\varphi$-modules).
\end{remark}

\begin{lemma}[{Dwork transfer theorem over $K[\s[t]\s]_0$}]\label{DTT}
The condition $R_1(M,1)=1$ is equivalent to $n=m$ (note that we assume that $K$ is discretely valued).
\end{lemma}
\begin{proof}
This is well-known to experts, however, we give a proof as a lack of appropriate references.

Assume that $n=m$. Then, $M\otimes K\{t\}\cong K\{t\}^m$ by Lemma \ref{abelian category}. Hence, we have $M\otimes K(t)_{\rho}\cong M\otimes K\{t\}\otimes K(t)_{\rho}\cong K(t)_{\rho}^m$ for $\rho\in (0,1)$, which implies that $R_1(M,\rho)=\rho$ for $\rho\in (0,1)$. Hence, $R_1(M,1)=\lim_{\rho\to 1-0}R_1(M,\rho)=1$ by \cite[Theorem 11.3.2]{pde}. To prove the converse, assume that $R_1(M,1)=1$. Then, $R_1(M,1)=\dots=R_m(M,1)=1$. Hence, $F_m(M,0)=0$. Recall that the function $[0,+\infty)\to\mathbb{R};r\mapsto F_m(M,r)$ is a convex, piecewise affine function with finitely many slopes (\cite[Theorem 11.3.2]{pde}); there exists $\alpha\in (0,1)$ such that $F_m(M,r)=mr$ for $r\in [\alpha,+\infty)$ by $p$-adic Cauchy theorem (\cite[Proposition 9.3.3]{pde}). Hence, the convexity implies that $F_m(M,r)=mr$ for $r\in [0,+\infty)$. Hence, $R_i(M,\rho)=\rho$ for $\rho\in (0,1)$ and $i=1,\dots,m$ by $R_1(M,\rho)\le\dots\le R_m(M,\rho)\le \rho$. By applying Dwork's transfer theorem (\cite[Theorem 9.6.1]{pde}) to $M\otimes K\langle t/\rho\rangle$ for all $\rho\in (0,1)$, the $K$-vector space $H^0(M\otimes K\{t/\rho\})$ for $\rho\in (0,1)$ has dimension $m$. We regard $H^0(M\otimes K\{t/\rho\})$ for $\rho\in (0,1)$ as a subspace of $H^0(M\otimes K[\s[t]\s])$. Then, $H^0(M\otimes K\{t/\rho\})=H^0(M\otimes K[\s[t]\s])$ for $\rho\in (0,1)$. Hence, $H^0(M\otimes K[\s[t]\s])=\cap_{\rho\in (0,1)}H^0(M\otimes K\{t/\rho\})=H^0(M\otimes \cap_{\rho\in (0,1)} K\{t/\rho\})=H^0(M\otimes K\{t\})$, which implies the assertion.
\end{proof}

\begin{thm}[{Theorem \ref{main result} (ii)}]\label{main result (ii)}
Assume that $n=m-1$. Then, Condition $(D)$ for $M$ holds. Consequently, Conjecture \ref{Dwork conjecture} for $M$ is true.
\end{thm}
\begin{proof}
By the assumption and Lemma \ref{DTT}, we have $R_1(M,1)<1$. Hence, the assertion follows from Theorem \ref{key lemma}.
\end{proof}

\begin{corollary}\label{two}
When $m\le 2$, Condition $(D)$ for $M$ holds.
\end{corollary}
\begin{proof}
By Remark \ref{remark on D}, we may assume that $(m,n)=(2,1)$, in which case the assertion is a special case of Theorem \ref{main result (ii)}.
\end{proof}

\end{document}